\documentclass[11pt]{amsproc}
\usepackage{comment}
\usepackage[small,bf]{caption}
\usepackage{nicefrac,color}
\usepackage{hyperref}
\hypersetup
{
    colorlinks,	%
    citecolor=blue,%
    filecolor=black,%
    linkcolor=blue,%
    urlcolor=blue
}
\usepackage{amssymb,amsmath,amsthm}
\usepackage{eulervm, xfrac}
\usepackage{amscd}
\usepackage{mathrsfs}
\usepackage{rotating, setspace}
\DeclareMathAlphabet{\mathpzc}{OT1}{pzc}{m}{it}
\usepackage[all, 2cell]{xy}
\UseTwocells

\makeatletter
\renewcommand\@seccntformat[1]{\csname the#1\endcsname.\quad}
\renewcommand\numberline[1]{\hb@xt@\@tempdima{#1.\hfil}}

\newcommand{\define}[1]{\textbf{#1}}
\newcommand{\sheaf}[1]{#1}
\newcommand{\cosheaf}[1]{\widehat{#1}}

\theoremstyle{definition}
\newtheorem{thm}{Theorem}[section]
\newtheorem{lem}[thm]{Lemma}
\newtheorem{defn}[thm]{Definition}
\newtheorem{cor}[thm]{Corollary}
\newtheorem{prop}[thm]{Proposition}
\newtheorem{ex}[thm]{Example}

\newtheorem{rmk}[thm]{Remark}

\renewcommand{\sfdefault}{iwona}
\DeclareMathAlphabet{\mathbfsf}{\encodingdefault}{\sfdefault}{bx}{n}

\newcommand{\dd}{\dagger}

\newcommand{\op}{\mathsf{op}}
\newcommand{\Hom}{\mathrm{Hom}}

\newcommand{\Set}{\mathbfsf{Set}}

\newcommand{\Vect}{\mathbfsf{Vect}}
\newcommand{\vect}{\mathbfsf{vect}}

\newcommand{\Fun}{\mathbfsf{Fun}}

\newcommand{\Open}{\mathbfsf{Open}}
\newcommand{\Closed}{\mathbfsf{Closed}}

\newcommand{\Shv}{\mathbfsf{Shv}}
\newcommand{\Coshv}{\mathbfsf{CoShv}}

\newcommand{\Alex}{\mathbfsf{Alex}}

\newcommand{\cat}{\mathbfsf{C}}
\newcommand{\dat}{\mathbfsf{D}}

\newcommand{\equivpc}{\mathcal{P}}

\newcommand{\covU}{\mathcal{U}}

\newcommand{\hG}{\widehat{G}}

\newcommand{\equivp}{\widehat{\mathcal{P}}}

\newcommand{\skycshv}{\widehat{S}}

\newcommand{\Lan}{\mathrm{Lan}}
\newcommand{\Ran}{\mathrm{Ran}}
\newcommand{\id}{\mathrm{id}}
\newcommand{\colim}{\varinjlim}

\newcommand{\cok}{\mathrm{cok}}
\newcommand{\im}{\mathrm{im}}

\newcommand{\supp}{\mathrm{supp}}

\newcommand{\RR}{\mathbb{R}}

\newtheoremstyle{justin}
{7pt}
{7pt}
{}
{7pt}
{\bf}
{:}
{.5em}
{}

\newcommand{\squigrightarrow}{\rightsquigarrow}

\begin{document}

\title{Dualities between Cellular Sheaves and Cosheaves}

\author{Justin Michael Curry}


\maketitle

\begin{abstract}
This paper affirms a conjecture of MacPherson: that the derived category of cellular sheaves is equivalent to the derived category of cellular cosheaves.
We give a self-contained treatment of cellular sheaves and cosheaves and note that certain classical dualities give rise to an exchange of sheaves with cosheaves.
Following a result of Pitts that states that cosheaves are cocontinuous functors on the category of sheaves, we use the derived equivalence provided here to gain a novel description of compactly supported sheaf cohomology.
\end{abstract}

\section{Introduction}

The utility of sheaf theory in mathematics is manifest.
What is less obvious is the utility of cosheaves, where data is assigned to open sets in a covariant way, i.e.~a cosheaf assigns to a pair of open sets $U\subseteq V$ a pair of objects and a morphism between them $\cosheaf{F}(U)\to\cosheaf{F}(V)$.
The relative absence of cosheaves from the literature is surprising since the simplest possible topological invariant, $\pi_0$, is a cosheaf~\cite{de2016categorified,woolf}.
Moreover, whenever we are given a sheaf $F$ the association that takes an open set to the derived group of compactly supported sections is also a cosheaf~\cite{borel1960homology,dt-lag}:
\[
U \qquad \squigrightarrow \qquad R\Gamma_c(U;F)
\]
This paper proves that in the special case of cellular sheaves and cosheaves the above association establishes a derived equivalence of categories, thereby proving a conjecture of Robert MacPherson. We note that similar results have been obtained by Schneider~\cite{schneider-vd} and Lurie~\cite{lurie-VD}. We now detail the contents of the paper.

Because of the sparse literature on cosheaves, a brief treatment of sheaves and cosheaves over a general topological space is given in Section~\ref{sec:general-sheaves}.
Thinking of an open cover as a functor from the nerve allows us to phrase the (co)sheaf axiom simply as commuting with (co)limits of this type.
Moreover, since the association of an element of the nerve with its corresponding open set is order-reversing, one is prepared for the apparent ``wrong direction'' of the arrows that appear in the definition of a cellular (co)sheaf.

In Section~\ref{sec:sheaves-on-posets}, we restrict our attention to sheaves and cosheaves over posets equipped with the Alexandrov topology.
The use of Kan extensions in this section is integral, and the association of elements to closed sets provides a topological explanation for one of the other pushforward operations that is defined for sheaves on posets in Section~\ref{subsec:open-pushforward}.

Cellular sheaves and cosheaves are introduced in Section~\ref{sec:cell-sheaves}.
Note that by using the classification result detailed in~\cite{curry-patel-CCC}, cellular (co)sheaves are equivalent to constructible (co)sheaves where the stratification is given by a cell structure.
However, the perspective used here is that cellular (co)sheaves are simply (co)sheaves on the face-relation poset of a cell complex.
Even with this more narrow interpretation, our duality result allows us to say that the face-relation poset of a cell complex is derived equivalent to its opposite poset, thus extending the results of Ladkani~\cite{ladkani2008derived}.

A description of projective and injective objects for these categories is given in Section~\ref{sec:derived_sheaf_cohom}. Although in general the category of sheaves does not have enough projectives (cf. Section~\ref{subsec:no-proj}), in the cellular setting there are both enough injectives and enough projectives.
This property is exploited in Section~\ref{sec:derived-homology} to define four homology theories, some of which are related through duality.
The derived perspective is also used to prove that these theories are invariant under subdivision.

The proof of the derived equivalence is given in Section~\ref{sec:derived-equivalence}, where an explicit formula is given that takes in a cellular sheaf and produces a complex of cellular cosheaves.
By applying linear duality, one obtains the formula for Verdier duality given by Shepard~\cite{shepard}.
Some apparently classical results are proved equite easily using this formulation of duality.

Section~\ref{sec:coend} concludes the paper with a categorical justification for the study of cosheaves. Using the derived equivalence proved here, compactly supported cellular sheaf cohomology can be described explicitly as a coend, or tensor product, of a sheaf with a particular complex of cosheaves.

\section{Sheaves and Cosheaves in General}
\label{sec:general-sheaves}

For this section, we assume that $X$ is a topological space and that $\dat$ is a category with enough limits and colimits. 
The collection of open sets of $X$ is denoted $\Open(X)$, which is a category with open sets for objects and inclusions for morphisms.
A collection of open sets $\covU:=\{U_i\}_{i\in\Lambda}$ covers an open set $U$ if $\cup_{i\in\Lambda}U_i=U$. 
A cover can alternatively be thought of as a functor from a particular category to $\Open(X)$ whose colimit is $U$, as we now explain.

\begin{defn}
Suppose $\covU:=\{U_i\}_{i\in \Lambda}$ is an open cover of $U$. 
The \define{nerve} of the cover $\covU$ is a category $N(\covU)$, whose objects are finite subsets 
\[
I=\{i_0,\ldots,i_n\} \subseteq \Lambda \qquad \text{such that} \qquad U_I:=U_{i_0}\cap\cdots\cap U_{i_n}\neq\emptyset,
\] 
with a unique arrow from $I\to J$ if and only if $I\subseteq J$.
\end{defn}

\begin{rmk}
Since we only consider finite subsets $I\subseteq \Lambda$ the corresponding sets $U_I$ are open. 
The association of $I$ to $U_I$ is order-reversing. 
This implies that we have the following functor and its formal dual:
\[
\iota_{\covU}:N(\covU)^{\op}\to\Open(X) \qquad \text{and} \qquad \iota_{\covU}^{\op}:N(\covU)\to \Open(X)^{\op}
\]
Note that $\covU$ covers $U$ if and only if the colimit of $\iota_{\covU}$ is $U$ and the limit of $\iota_{\covU}^{\op}$ is $U$. 
\end{rmk}

\begin{defn}
Any functor $\sheaf{F}:\Open(X)^{\op}\to\dat$ is a \define{presheaf}. We refer to the morphism $\rho_{V,U}^F:F(U)\to F(V)$ associated to the inclusion $U\subseteq V$ as the \define{restriction map} from $V$ to $U$.
Dually, any functor $\cosheaf{F}:\Open(X)\to\dat$ is a \define{precosheaf}.
The morphism $r^{\cosheaf{F}}_{U,V}:\cosheaf{F}(V)\to\cosheaf{F}(U)$ associated to the inclusion $U\subseteq V$ is the \define{extension map} from $U$ to $V$.
We often drop the superscript for convenience.
\end{defn}

\begin{ex}[Leray Presheaves]\label{ex:Leray-presheaves}
Suppose $f:Y\to X$ is a continuous map. 
The functor
\[
U\subseteq X \qquad \squigrightarrow \qquad H^i(f^{-1}(U);\Bbbk),
\]
where $H^i(-;\Bbbk)$ is singular cohomology in degree $i$ with coefficients in a field $\Bbbk$, is a presheaf of vector spaces.
These are the \define{Leray presheaves} of $f$. 
\end{ex}

\begin{ex}[Supported Functions]\label{ex:supp-precosheaves}
The functor
\[
U\subseteq X \qquad \squigrightarrow \qquad \{f:U\to\RR\,|\, \text{supp}(f)\subset U \},
\]
where $\supp(f)$ is the closure of the set of points where $f$ is non-zero within $X$, is a precosheaf of sets.
If $f$ is supported on $U$, then it is also supported on any open set $V$ containing $U$, via extension by zero.
\end{ex}

\begin{defn}[Sheaf and Cosheaf Axioms]
Let $\sheaf{F}$ and $\cosheaf{F}$ denote a presheaf and a precosheaf, respectively.
We say $\sheaf{F}$ is a \define{sheaf} if it commutes with limits of type $\iota_{\covU}^{\op}$, i.e.~the universal arrow
\[
\sheaf{F}(U)=\sheaf{F}(\varprojlim \iota^{\op}_{\covU}) \longrightarrow \varprojlim (\sheaf{F}\circ\iota_{\covU}^{\op})=:F[\covU]
\]
is an isomorphism for every possible cover $\covU$.
Dually, we say $\cosheaf{F}$ is a \define{cosheaf} if it commutes with colimits of type $\iota_{\covU}$, i.e.~the universal arrow
\[
\cosheaf{F}[\covU]:=\varinjlim (\cosheaf{F}\circ \iota_{\covU}) \longrightarrow \cosheaf{F}(\varinjlim \iota_{\covU})=\cosheaf{F}(U)
\]
is an isomorphism for every possible cover $\covU$.
If the above maps are isomorphisms for a particular cover $\covU$, we say it is a (co)sheaf on the cover.
\end{defn}

\begin{ex}[Sheaf of Sections]
Given a continuous map $f:Y\to X$, the functor
\[
U\subseteq X \qquad \squigrightarrow \qquad \{s:U\to Y \,|\,f\circ s=\id_U\},
\]
is the canonical example of a sheaf of sets.
It is called the \define{sheaf of sections} of $f$.
\end{ex}

\begin{ex}[Reeb Cosheaf]
Given a continuous map $f:Y\to X$, the functor
\[
U\subseteq X \qquad \squigrightarrow \qquad \pi_0(f^{-1}(U)),
\]
where $\pi_0$ is the set of path components, is the canonical example of a cosheaf of sets.
It is called the \define{Reeb cosheaf} of $f$.
\end{ex}

\begin{ex}[Need of Sheafification]
This is a continuation of Example~\ref{ex:Leray-presheaves}.
The functor
\[
U\subseteq X \qquad \squigrightarrow \qquad H^i(f^{-1}(U);\Bbbk)
\]
is a sheaf of vector spaces for $i=0$. For $i>0$ and for the covers $\covU$ consisting of two elements $U$ and $V$, the failure of the sheaf axiom is measured by the connecting homomorphism in the Mayer-Vietoris long exact sequence. 
For more complicated covers a spectral sequence is needed.
However, every presheaf of vector spaces can be \emph{sheafified} and the sheafification of the Leray presheaves define the Leray sheaves.
\end{ex}

\begin{ex}[Partitions of Unity]
Suppose $X$ is paracompact.
The functor 
\[
U\subseteq X \qquad \squigrightarrow \qquad \{f:U\to\RR\,|\, \text{supp}(f)\subset U \},
\]
is not a cosheaf of sets, but it is a cosheaf of $\RR$-modules, by virtue of the existence of a partition of unity for any cover $\covU$.
\end{ex}

\subsection{Refinement of Covers}

We now describe how the sheaf or cosheaf axiom is inherited by coarser covers.

\begin{defn}[Refinement of Covers]
  Suppose $\covU$ and $\covU'$ are covers of $U$, then we say that $\covU'$ \define{refines} $\covU$ if for every $U_i'\in\covU'$ there is a $U_j\in\covU$ and an inclusion $U_i'\to U_j$. Note that every cover refines the trivial cover $\{U\}$.
\end{defn}

\begin{lem}
  Let $\cosheaf{F}$ and $F$ be a precosheaf and a presheaf respectively. Suppose $\covU'$ refines another cover $\covU$ of an open set $U$, then there are well-defined maps
  \[
  \cosheaf{F}[\covU']\to \cosheaf{F}[\covU] \qquad \mathrm{and} \qquad F[\covU]\to F[\covU'].
  \]
\end{lem}
\begin{proof}
We'll detail the proof for a pre-cosheaf $\cosheaf{F}$. A refinement $\covU'\to\covU$ defines a natural transformation $\cosheaf{F}\circ \iota_{\covU'}\Rightarrow \cosheaf{F}\circ \iota_{\covU}$. The colimit defines a natural transformation from $\cosheaf{F}\circ \iota_{\covU}$ to the constant diagram $\cosheaf{F}[\covU]$. Since the composition of natural transformations is a natural transformation, this induces a cocone $\cosheaf{F}\circ\iota_{\covU'}\Rightarrow \cosheaf{F}[\covU]$ which, by the universal property of the colimit, defines a unique induced map there, i.e.
\[
\cosheaf{F}\circ \iota_{\covU'}\Rightarrow \cosheaf{F}\circ \iota_{\covU}\Rightarrow \cosheaf{F}[\covU] \quad \mathrm{implies} \quad \exists!\, \cosheaf{F}[\covU']\to\cosheaf{F}[\covU].
\]
It is easily checked that this map is independent of the choices made when defining the refinement.
\end{proof}

\begin{cor}\label{cor:refined_axiom}
  If $\cosheaf{F}$ is a cosheaf or $F$ is a sheaf for the cover $\covU'$, then it is a cosheaf or sheaf for every cover it refines.
\end{cor}
\begin{proof}
  Suppose we have a series of refinements
  \[
  \covU'\to\covU\to \{U\}.
  \] 
  To say that $\cosheaf{F}$ or $F$ is cosheaf or sheaf for $\covU'$ is to say that the following induced maps are isomorphisms:
  \[
  \xymatrix{\cosheaf{F}[\covU'] \ar[r] \ar@/^2pc/[rr]^{\cong} & \cosheaf{F}[\covU] \ar[r] & \cosheaf{F}(U)} \qquad \xymatrix{F(U) \ar[r] \ar@/^2pc/[rr]^{\cong} & \cosheaf{F}[\covU] \ar[r] & \cosheaf{F}[\covU']}
  \]
  However, by functoriality, the factored maps must themselves be isomorphisms, i.e.
  \[
  \xymatrix{\cosheaf{F}[\covU]\ar[r]^{\cong} & \cosheaf{F}(U)} \qquad \xymatrix{F(U) \ar[r]^{\cong} & F[\covU]}.
  \]
\end{proof}

We will make use of this corollary when studying sheaves and cosheaves on a poset equipped with the Alexandrov topology.

\section{Sheaves and Cosheaves on Posets}
\label{sec:sheaves-on-posets}

In this section we specialize the discussion of Section~\ref{sec:general-sheaves} to spaces associated to posets.
Pavel Alexandrov defined a topology associated to a finite poset in~\cite{alex-dr}.
Many have considered sheaves over posets, e.g.~\cite{ladkani2008homological,dihom_1}, but we point out the obvious dualizations for cosheaves along the way.
For this section $X$ will denote a partially ordered set $(X,\leq_X)$.
As before, $\dat$ is a category with enough limits and colimits.

\begin{defn}[Alexandrov Topology]
The \define{Alexandrov topology} is defined as follows: $U\subseteq X$ is open if and only if it is an up-set, i.e.
\[
x\in U \quad \mathrm{and} \quad x\leq y \qquad \mathrm{implies} \qquad y\in U.
\] 
Let $\Alex(X)$ denote the category of Alexandrov opens ordered by inclusion. 
The smallest open set containing an element $x\in X$ is the \textbf{open star} or \define{basic open} of $x$, written $U_x:=\{y\in X\,|\, x\leq y\}$. 
Additionally, we define the \define{closure of x} by $\bar{x}:=\{y\in X \,|\, y\leq x\}$.
\end{defn}

\begin{ex}\label{ex:closed_gen}
  Consider $(\RR,\leq)$ with the usual partial order. The open sets are all those open or half open intervals such that the right-hand endpoint is $+\infty$. Observe that the closed set $(-\infty,0)$ cannot be written as an intersection of closed sets of the form $\bar{t}$. Thus the closures at $t$ do not form a basis.
\end{ex}

Observe that, just as with the nerve, the association of elements of $X$ to their corresponding basic opens is order reversing, i.e.~we have a functor 
\[
\iota:X^{\op} \to \Alex(X) \qquad \text{since} \qquad x\leq y \quad \Rightarrow \quad U_y\subseteq U_x.
\]
This is what motivates the next proposition.

\begin{prop}\label{prop:equivalence}
If $X$ is a poset then any functor $\sheaf{F}:X\to\dat$ or $\cosheaf{F}:X^{\op}\to\dat$ defines a sheaf or cosheaf on the Alexandrov topology via a right or left Kan extension, respectively.
\[
\xymatrix{ X \ar[r]^{\sheaf{F}} \ar[d]_{\iota} & \dat \\ \Alex(X)^{\op} \ar@{.>}[ru]_{\Ran_{\iota} \sheaf{F}} & }
\qquad 
\xymatrix{ X^{\op} \ar[r]^{\cosheaf{F}} \ar[d]_{\iota} & \dat \\ \Alex(X) \ar@{.>}[ru]_{\Lan_{\iota} \cosheaf{F}} & }
\]
\end{prop}

\begin{proof}
We recall what these Kan extensions assign to a given Alexandrov open $U$ where we make the the notational identifications $\sheaf{F}=\Ran_{\iota} \sheaf{F}$ and $\cosheaf{F}=\Lan_{\iota} \cosheaf{F}$:
\[
\sheaf{F}(U):= \varprojlim_{x\in U}\sheaf{F}(x) \qquad \text{and} \qquad \cosheaf{F}(U)\cong \varinjlim_{x\in U}\cosheaf{F}(x)
\]
The extension forces the (co)sheaf axiom for the finest possible cover of $U=\cup_{x\in U} U_{x}$, which by Corollary~\ref{cor:refined_axiom} implies the corresponding axiom for every other possible cover.
\end{proof}

\begin{ex}[Nerve]
Suppose $\sheaf{F}$ is a sheaf on a general topological space $X$. Suppose $\covU=\{U_i\}_{i\in\Lambda}$ is an open cover. The nerve is also a poset, so composing $\iota^{\op}_{\covU}:N(\covU)\to\Open(X)^{\op}$ with $\sheaf{F}:\Open(X)^{\op}\to\dat$ gives a sheaf over the nerve equipped with the Alexandrov topology.
\end{ex}

Let $[\cat,\dat]$ denote the category whose objects are functors $F:\cat\to\dat$ and whose morphisms are natural transformations. 
The following is a simple modification of an equivalence known to experts, which is normally stated for sheaves~\cite{ladkani2008derived}.

\begin{cor}\label{cor:equivalence}

The following categories are equivalent:
\[
[X,\dat]\cong \Shv(X;\dat) \qquad [X^{\op},\dat]\cong \Coshv(X;\dat)
\]
When $\dat=\Vect$ we will drop the extra notation and write $\Shv(X)$ and $\Coshv(X)$.
\end{cor}

Given this equivalence, one can ask how standard operations on sheaves, e.g.~pullback and pushforward, translate into operations on the categories $[X,\dat]$ and $[X^{\op},\dat]$.
Topological interpretations of these functors are provided as well, with special attention paid to the pushforward with open supports described in Section~\ref{subsec:open-pushforward}. 

\subsection{Pullback}
\label{subsec:pullback}

Suppose $(X,\leq_X)$ and $(Y,\leq_Y)$ are posets. 
A map of posets is a map of sets $f:X\to Y$ that is order-preserving, i.e. if $x\leq_X x'$ then $f(x)\leq_Y f(x')$. 
Alternatively, since a poset can be viewed as a category, a map of posets is just a functor. We will continue to abbreviate $(X, \leq_X)$ as $X$ and $(Y,\leq_Y)$ as $Y$.

\begin{defn}
  Given a sheaf $\sheaf{G}$ or a cosheaf $\cosheaf{G}$ on $Y$ and a map of posets $f:X\to Y$, we can define the \define{pullback} or \define{inverse image} to be the obvious pre-composition that completes the following diagrams:
\[
\xymatrix{Y \ar[r]^-{\sheaf{G}} & \dat \\
X \ar[u]^{f} \ar@{-->}[ru]_{f^*\sheaf{G}} & } 
\qquad
\xymatrix{Y^{\op} \ar[r]^-{\cosheaf{G}} & \dat \\
X^{\op} \ar[u]^{f^{\op}} \ar@{-->}[ru]_{f^*\cosheaf{G}} & } 
\]

Since both of these constructions are functorial, we have defined two functors:
  \[
  f^*:\Shv(Y;\dat) \to \Shv(X;\dat) \qquad f^*:\Coshv(Y;\dat)\to\Coshv(X;\dat).
  \]
\end{defn}

\begin{rmk}
The above definition agrees with the classical definition of the pushforward sheaf, as we now see. 
Recall that given a sheaf $\sheaf{G}$ on a topological space $Y$ and a continuous map $f:X\to Y$, the sheaf $f^*\sheaf{G}$ is defined to be the sheafification of the pre-sheaf
\[
  f^*G(U):=\varinjlim_{V\supset f(U)} G(V).
\]
However, when working in the Alexandrov topology
\[
  f^*\sheaf{G}(U_x):=\varinjlim_{V \supset f(U_x)} \sheaf{G}(V) \cong \sheaf{G}(V_{f(x)})=\sheaf{G}(f(x)),
\]
where we have used the fact that the smallest open set containing $f(U_x)=f(\{x' | x\leq x'\})$ is $V_{f(x)}=\{y|f(x)\leq y\}$.
\end{rmk}

\begin{ex}[Constant Sheaf and Cosheaf]
 Consider the constant map $p:X\to\star$. A sheaf $\sheaf{G}$ on $\star$ is completely specified by a vector space $W$, so we can refer to $\sheaf{G}$ by the name $W$. We define the constant sheaf on $X$ with value $W$ to be $W_X:=p^*W$. One sees that it is a sheaf that assigns $W$ to every cell with all the restriction maps being the identity. Similarly, the constant cosheaf with value $W$ is $\cosheaf{W}_X:=p^* W$.
\end{ex}

\subsection{Pushforward}
\label{subsec:pushforward}

\begin{defn}
Given a sheaf $\sheaf{F}$ or a cosheaf $\cosheaf{F}$ on $X$ and a map of posets $f:X\to Y$, the \define{pushforward (co)sheaf} $f_*\sheaf{F}$ is the right (left) Kan extension of $F$ along $f$, i.e. on point a $y\in Y$, $$f_*\sheaf{F}(y):=\varprojlim_{f(x)\geq y} F(x) \qquad f_*\cosheaf{F}(y)=\varinjlim_{f(x)\geq y} \cosheaf{F}(x)$$
Since both of these constructions are functorial, we have defined two functors:
  \[
  f_*:\Shv(X;\dat) \to \Shv(Y;\dat) \qquad f_*:\Coshv(X;\dat)\to\Coshv(Y;\dat)
  \]
\end{defn}

\begin{rmk}
Let us check that the above point-wise description of the pushforward agrees with the usual formulation of the pushforward. 
If $f:X\to Y$ is a map of topological spaces, then for every open set $V\subseteq Y$, the pre-image $f^{-1}(V)$ is open as well. 
Hence, the value of the pushforward (co)sheaf on $V$ is just the value on $f^{-1}(V)$.
In the Alexandrov topology, the pre-image of a basic open $V_y$ is just $f^{-1}(V_y)=\{x\in X\,|\, f(x)\geq y\}$. Taking the limit or colimit over this indexing set simply appeals to the sheaf or cosheaf axiom for the finest cover of this open set.
\end{rmk}

\begin{rmk}
Note that, in contrast to pullbacks, this definition is context-dependent. The operation $f_*$ is defined differently depending on whether it is being applied to a sheaf or a cosheaf.
\end{rmk}

\begin{ex}[Global Sections]
Let $p:X\to \star$ be the constant map. We have the following obvious isomorphisms, some of which will be proved later in the paper.
\[
  p_* \sheaf{F}(\star)\cong \sheaf{F}(X) \cong \varprojlim \sheaf{F} \cong H^0(X;\sheaf{F}) \qquad p_*\cosheaf{F}(\star) \cong \cosheaf{F}(X) \cong \varinjlim\cosheaf{F} \cong H_0(X;\cosheaf{F})
\]
As we will see, the derived functors of $p_*$ will be used to define higher sheaf cohomology and cosheaf homology.
\end{ex}

\subsection{Pushforward with Open Supports}
\label{subsec:open-pushforward}

Since we defined the pushforward of a sheaf on a poset using the right Kan extension, one can just as well ask what happens when we use the left Kan extension instead. This turns out to be the basis for defining \emph{sheaf homology} and \emph{cosheaf cohomology}.

\begin{defn}
Given a sheaf $\sheaf{F}$ or a cosheaf $\cosheaf{F}$ on $X$ and a map of posets $f:X\to Y$, the \define{pushforward with open supports} is the left Kan extension of $F$ along $f$ or the right Kan extension of $\cosheaf{F}$ along $f^{\op}$, i.e. on point a $y\in Y$, 
\[
f_{\dd}F(y)=\varinjlim_{f(x)\leq y} F(x) \qquad f_{\dd}\cosheaf{F}(y):= \varprojlim_{f(x)\leq y}\cosheaf{F}(x)
\]

Since both of these constructions are functorial, we have defined two functors:
  \[
  f_{\dd}:\Shv(X;\dat) \to \Shv(Y;\dat) \qquad f_{\dd}:\Coshv(X;\dat)\to\Coshv(Y;\dat)
  \]
\end{defn}

This functor appears to be quite unusual, despite its naturality from the categorical perspective. To explain its topological origin, we revisit some of the original ideas of Alexandrov.

When Alexandrov first defined his topology he did two things differently:
\begin{enumerate}
  \item He only defined the topology for \emph{finite} posets.
  \item He defined the \emph{closed} sets to have the property that if $x\in V$ and $x'\leq x$, then $x'\in V$.
\end{enumerate}
\index{cosheaf!on closed sets}\index{sheaf!on closed sets}
Let us repeat the initial analysis of sheaves and diagrams indexed over posets, where we now put closed sets on equal footing with open sets. Observe that as before we have an inclusion functor:
\[
 j:(X,\leq) \to \Closed(X) \qquad x \rightsquigarrow \bar{x}:=\{x'| x'\leq x\}
\]
Consequently, we have a similar diagram for a functor $F:X\to\dat$ as before.
\[
  \xymatrix{X \ar[r]^F \ar[d]_{j} & \dat \\
  \Closed(X) \ar@{.>}[ur]_{?} & }
\]

If we choose the left Kan extension, we'd like to say the extended functor is a cosheaf on closed sets, i.e. use the definition of a cosheaf but replace open sets with closed sets. Unfortunately, this concept is not well defined for general topological spaces because the arbitrary union (colimit) of closed sets is not always closed. For Alexandrov spaces this property does hold and this illustrates one of the extra symmetries this theory possesses.

However, in order for the Kan extension to take a diagram and make it into a cosheaf, we need to know whether the image of the inclusion functor defines a basis for the closed sets. In Example \ref{ex:closed_gen} we showed that this is not always the case. The topology generated by the image of this functor is called the \textbf{specialization} topology and it suffers from certain technical deficiencies. In particular, order-preserving maps are not necessarily continuous in this topology, thus it fails to give a functorial theory. Fortunately, for finite posets these topologies agree and we can talk about cosheaves on closed sets without any trouble.

We now can give a topological explanation for the existence of the functor $f_{\dd}$. It is the functor analogous to ordinary pushforward where we have adopted closed sets as the indexing category for cosheaves and sheaves. If $f:X\to Y$ is a map of posets, then $f^c$ is the induced map between closed sets. The dagger pushforward is then the obvious completion of the diagram.
\[
\xymatrix{\Closed(X) \ar[r]^-F & \dat \\
\Closed(Y) \ar[u]^{f^c} \ar@{-->}[ru]_{f_{\dd}F} & }
\]

In Section \ref{sec:derived-homology} this functor provides the foundation for defining \textbf{sheaf homology} and \textbf{cosheaf cohomology}, which are theories that don't exist for general spaces.

\subsection{Adjunctions}
\label{subsec:adjunctions}

Finally, we note the usual adjunctions between these functors hold.

\begin{thm}
The functors $f^*:\Shv(Y)\to \Shv(X)$ and $f_*:\Shv(X)\to\Shv(Y)$ form an adjoint pair $(f^*,f_*)$ and thus
\[
\Hom_{\Shv(X)}(f^*G,F)\cong \Hom_{\Shv(Y)}(G,f_*F).
\]
Dually, the functors for cosheaves satisfy the opposite adjunction $(f_*,f^*)$
\[
\Hom_{\Coshv(Y)}(f_*\cosheaf{F},\hG) \cong \Hom_{\Coshv(X)}(\cosheaf{F},f^*\hG).
\]
\end{thm}
\begin{proof}
  Recall that $f^*(f_*F)(x)=(f_*F)(f(x))$. Using the fact that 
  $$(f_*F)(f(x)) = \varprojlim\{F(z)|f(z)\geq f(x)\},$$ 
  we get a map to $F(x)$ since $x\in f^{-1}(f(x))$ and this morphism is final for each $x$. This implies there is a natural transformation of functors $f^*\circ f_*\to\id$, which is universal (final). 
  
  Similarly, 
$$f_*(f^*G)(y)=\varprojlim\{f^*G(x)=G(f(x))|f(x)\geq y\}$$ 
  and since $y\leq f(x)$ we can use the restriction map $\rho^G_{f(x),y}:G(y)\to G(f(x))$. The universal property of the limit guarantees a map $G(y)\to\varprojlim G(f(x))=f_*f^*G(y)$ and thus a natural transformation of functors $\id\to f_*f^*$.
\end{proof}

\begin{thm}
  The functors $f_{\dd}:\Shv(X)\to\Shv(Y)$ and $f^*:\Shv(Y)\to\Shv(X)$ form an adjoint pair $(f_{\dd},f^*)$ and thus
  \[
  \Hom_{\Shv(Y)}(f_{\dd}F,G)\cong\Hom_{\Shv(X)}(F,f^*G).
  \]
  Dually, the functors for cosheaves satisfy the opposite adjunction $(f^*,f_{\dd})$
  \[
  \Hom_{\Coshv(X)}(f^*\hG,\cosheaf{F}) \cong \Hom_{\Coshv(Y)}(\hG,f_{\dd}\hG).
  \]
\end{thm}
\begin{proof}
  $f_{\dd}(f^*G)(y)=\varinjlim \{G(f(x))|f(x)\leq y\}$ so again we can use the restriction maps to define maps to $G(y)$. The universal property of colimits gives a map $f_{\dd}f^*G(y)\to G(y)$ and thus a map of functors $f_{\dd}f^*\to\id$. Similar arguments give a map $\id\to f^*f_{\dd}$
\end{proof}

\section{Cellular Sheaves and Cosheaves}
\label{sec:cell-sheaves}

The sheaves of interest in this paper were explicitly outlined by Robert MacPherson.
Subsequent work on cellular sheaves was carried out in the theses of Shepard~\cite{shepard}, Vybornov~\cite{vybornov-triang}, and the author~\cite{curry2014sheaves}.
Shepard provided the first truly detailed account of the subject, but his thesis was never published.
We review the necessary definitions.

\begin{defn}\label{defn:reg-cell}
  A \define{regular cell complex} is a space $|X|$ equipped with a partition into pieces $\{|\sigma|\}_{\sigma\in X}$ called \define{cells} satisfying the following four properties. We denote the closure of a cell $|\sigma|\subseteq|X|$ by $|\Bar{\sigma}|$.
  \begin{enumerate}
    \item Each point $x\in |X|$ has an open neighborhood intersecting only finitely many $|\sigma|$.
    \item $|\sigma|$ is homeomorphic to $\RR^k$ for some $k$ (where $\RR^0$ is one point).
    \item If $|\Bar{\tau}|\cap |\sigma|$ is non-empty, then $|\sigma|\subseteq |\Bar{\tau}|$.
    \item The pair $|\sigma|\subset |\Bar{\sigma}|$ is homeomorphic to the pair $B^{k}\subset \bar{B}^k$, i.e.~there is a homeomorphism from the closed ball $\varphi:\bar{B}^k\to |\Bar{\sigma}|$ that sends the interior to $|\sigma|$.
  \end{enumerate}
  We will refer to the space and its indexing set as a pair $(|X|,X)$.
\end{defn}

\begin{defn}\label{defn:cell-cmplx}
 A \define{cell complex} is a space $X$ with a partition into pieces $\{|\sigma|\}_{\sigma\in X}$ whose one point compactification is a regular cell complex.
\end{defn}

\begin{ex}
The geometric realization of a simplicial complex $|K|$ is a regular cell complex, where the \emph{interior} of a simplex is regarded as a cell.
In fact, every regular cell complex can be refined so that each cell is the interior of a triangulation~\cite[III.1.7]{lundell1969topology}.
\end{ex}

\begin{ex}[Non-example]
 The open interval $(0,1)$ decomposed with only one open cell is \emph{not} a cell complex. Its one-point compactification is the circle decomposed with one vertex $\{\infty\}$ and one edge $(0,1)$ whose attaching map is not an embedding, thus contradicting the fourth axiom. Further refining the decomposition of $(0,1)$ to $(0,\frac{1}{2})\cup\{\frac{1}{2}\}\cup(\frac{1}{2},1)$ gives an example of a cell complex. 
\end{ex}

\begin{defn}
The third property of a cell complex $(|X|,X)$ endows $X$ with the structure of a poset, where $\sigma\leq\tau$ if and only if $|\sigma|\subseteq \Bar{|\tau|}$. 
This is the \define{face relation poset} of $(|X|,X)$.
Note that this poset is naturally graded by dimension, where the dimension of $\sigma$ refers to the dimension of the open ball $B^i$ that $|\sigma|$ is homeomorphic to. 
We will write $\sigma^i$ if the dimension of $\sigma$ is $i$. 
The set $X^i$ will be the subset of $X$ of dimension $i$ cells.
\end{defn}

\begin{rmk}\label{rmk:quotient-topology}
The Alexandrov topology on the face-relation poset arises quite naturally.
Consider the quotient topology on $|X|$ that identifies two points $x\sim x'$ if they belong to the same cell $|\sigma|$.
This defines a map
\[
  \xymatrix{|X| \ar[d]^{q} \\ X:=|X|/\sim}
\]
that is continuous if we declare the open stars $U_{\sigma}:=\{\tau\in X\,|\, \sigma\leq\tau\}$ to be open. 
See Figure~\ref{fig:alex_interval} for an example.
\end{rmk}

\begin{figure}[ht]
  \centering
  \includegraphics[width=.5\textwidth]{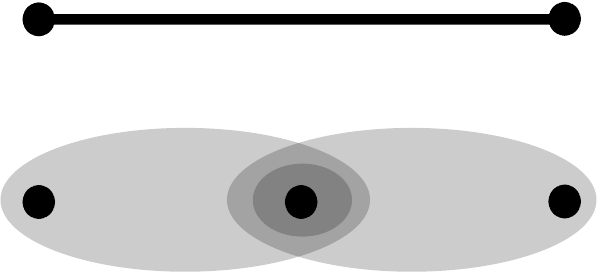}
  \caption{Interval with three cells and associated Alexandrov space}
  \label{fig:alex_interval}
\end{figure}

\begin{defn}\label{defn:cell-sheaves}
A \define{cellular sheaf} $\sheaf{F}$ on $X$ is a functor $\sheaf{F}:X\to\dat$.
Similarly, a \define{cellular cosheaf} $\cosheaf{F}$ is a functor $\cosheaf{F}:X^{\op}\to\dat$.   
Understanding that Corollary~\ref{cor:equivalence} applies, we will abuse notation and refer to the functor categories $[X,\dat]$ and $[X^{\op},\dat]$ as $\Shv(X;\dat)$ and $\Coshv(X;\dat)$, respectively.
Unless otherwise indicated, we will assume that $\dat=\Vect_{\Bbbk}$ and write $\Shv(X)$ and $\Coshv(X)$ for short.
\end{defn}

\begin{rmk}
Unbeknownst to MacPherson, Sir Christopher Zeeman outlined a similar definition in his thesis~\cite{dihom_1}, but used the term ``stack.''
Zeeman's choice of terminology was unfortunate, since his term is now used for a more general concept.
\end{rmk}

\begin{rmk}\label{rmk:3reasons}
There are multiple reasons why these functors can be viewed as (co)sheaves.
The simplest reason---and the one detailed in Proposition~\ref{prop:equivalence}---is that these functors define (co)sheaves on the Alexandrov topology of the face-relation poset.
A second reason is that if a sheaf is locally constant when restricted to each cell, then one can pushforward and pullback this sheaf along the map $q:|X|\to X$ described in Remark~\ref{rmk:quotient-topology} to obtain an isomorphic sheaf.
A third reason is that (co)sheaves that are constructible with respect to a cell structure are equivalent to cellular (co)sheaves~\cite{AFT,curry-patel-CCC,kash-rh,treumann-stacks}.
\end{rmk}

\begin{ex}
Consider the unit interval $|X|=[0,1]$ with the following cell structure: $|x|=0$, $|y|=1$, and $|a|=(0,1)$. 
Then a cellular sheaf (left) or a cellular cosheaf (right) on this cell complex is completely determined by choosing three vector spaces and two linear maps for each incidence relation.
\[
  \xymatrix{& F(a) & \\ F(x) \ar[ur]^{\rho_{a,x}} & & \ar[lu]_{\rho_{a,y}} F(y)}
  \qquad
  \xymatrix{& \cosheaf{F}(a) \ar[dl]_{r_{x,a}} \ar[dr]^{r_{y,a}} & \\ \cosheaf{F}(x) & & \cosheaf{F}(y)}
\]
\end{ex}

One reason for studying cellular sheaves and cosheaves simultaneously is that studying a celluar sheaf over a triangulated manifold automatically yields a cellular cosheaf over the dual triangulation.
See Figure~\ref{fig:dual-sheaf-cosheaf}.

\begin{prop}[Poincar\'e Duality]\label{thm:mfld_sheaf_cosheaf}
 Suppose $F$ is a cellular sheaf on a triangulated closed $n$-manifold $X$, then $F$ defines a cellular cosheaf $\cosheaf{F}$ on the dual triangulation $\tilde{X}$.
\end{prop}
\begin{proof}
Let a superscript tilde indicate the dual cell, i.e.~$\tilde{\sigma}$ is the cell dual to $\sigma$.
If we have a pair $\sigma\leq\tau$ in $X$, then $\tilde{\tau}\leq\tilde{\sigma}$ in $\tilde{X}$.
We can now define the cellular cosheaf $\cosheaf{F}$ as follows.
On dual cells, we can simply define $\cosheaf{F}(\tilde{\sigma}):=F(\sigma)$ and $\cosheaf{F}(\tilde{\tau}):=F(\tau)$.
The extension maps are also defined by $r_{\tilde{\tau},\tilde{\sigma}}:=\rho_{\tau,\sigma}$.
Said diagramatically, we have the following:
\[
 \xymatrix{F(\sigma^i) \ar[r]^{\rho_{\tau,\sigma}} \ar@{=}[d] & F(\tau^{i+1}) \ar@{=}[d] \\
\cosheaf{F}(\tilde{\sigma}^{n-i}) \ar[r]^{r_{\tilde{\tau},\tilde{\sigma}}} & \cosheaf{F}(\tilde{\tau}^{n-i-1})}
\]
Said simply, the same abstract diagram of vector spaces $F:X\to\Vect$ defines a diagram over $\cosheaf{F}:\tilde{X}^{\op}\to\Vect$, i.e. a cosheaf on the dual cell structure.
\end{proof}


\begin{figure}[ht]
  \centering
  \includegraphics[width=.5\textwidth]{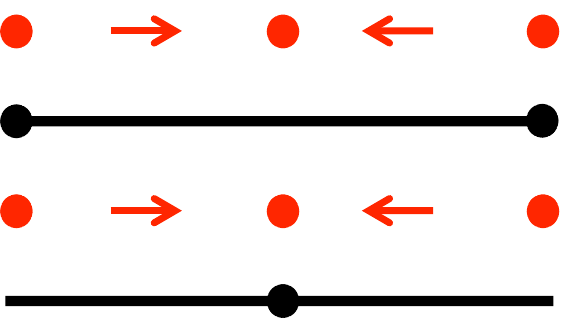}
  \caption{A cellular sheaf defines a cellular cosheaf over the dual cell structure}
  \label{fig:dual-sheaf-cosheaf}
\end{figure}
\subsection{Cellular Sheaf Cohomology}
\label{subsec:cell-sheaf-cohomology}

One of the pleasant features of cellular sheaves is that their cohomology is no more difficult to compute than ordinary cellular cohomology of a cell complex. 
One simply replaces cellular cochains valued in the field $\Bbbk$ with cellular cochains valued in the collection of $\Bbbk$-vector spaces $\{F(\sigma)\}$.
One can view the formulas provided below as simplified versions of formula that \v{C}ech cohomology would provide.


Let's recall one of the basic facts about cell complexes. We write $\sigma\leq_i\tau$ if the difference in dimension between $\tau$ and $\sigma$ is $i$.

\begin{lem}\label{lem:2cells}
If $\sigma\leq_2\tau$, then there are exactly two cells $\lambda_1,\lambda_2$ where $\sigma\leq_1\lambda_i\leq_1\tau$. Said diagrammatically:
\[
  \xymatrix{ & \tau & \\ \lambda_1 \ar[ru] & & \lambda_2 \ar[lu]\\ & \ar[lu] \sigma \ar[ru] &}
\]
\end{lem}


\begin{proof}
The statement is well known for a regular cell complex~\cite[p.~31]{cooke1967homology}.
Now, if $(|X|,X)$ is a cell complex, then it has a regular one-point compactification and neither cell $\lambda_1$ nor $\lambda_2$ can be the point at infinity if $\sigma$ is in $X$.
\end{proof}

As is the case in cellular homology, we need a sign condition that distinguishes these two different sequences of incidence relations.

\begin{defn}\label{defn:signed-reln}
 A \textbf{signed incidence relation} is an assignment to any pair of cells $\sigma,\tau\in X$ a number $[\sigma:\tau]\in\{0,\pm 1\}$ such that 
 \begin{itemize}
 \item if $[\sigma:\tau]\neq 0$, then $\sigma\leq_1 \tau$, and
 \item if $\gamma$ and $\tau$ are any pair of cells, the sum $\sum_{\sigma}[\gamma:\sigma][\sigma:\tau]=0$.
 \end{itemize}
\end{defn}

One way to get such a signed incidence relation is to choose a local orientation (via the homeomorphism of each cell $|\sigma|$ with $\RR^k$) for each cell without regard to global consistency. Lemma~\ref{lem:2cells} then guarantees that the second property of Definition~\ref{defn:signed-reln} is satisfied.

We can now provide formulae for computing cellular sheaf cohomology and cellular cosheaf homology that is completely analogous to cellular cohomology.

\begin{defn}[\cite{shepard,dihom_1}]\label{defn:cpt-cohom}
 Given a cellular sheaf $F:X\to\Vect$ we define its \textbf{compactly supported co-chains} in degree $n$ to be the product of the vector spaces residing over all the $n$-dimensional cells.
\[
 C^n_c(X;F)=\bigoplus_{\sigma\in X^n}F(\sigma)
\]
 These vector spaces are graded components in a complex of vector spaces $C^{\bullet}_c(X;F)$. The differentials are defined by
\[
 \delta^n_c=\sum_{\sigma\leq\tau} [\sigma^n:\tau^{n+1}]\rho_{\tau,\sigma}.
\]
The cohomology of the complex $C^{\bullet}_c(X;F)$
\[
 \xymatrix{0 \ar[r] & \oplus F(\mathrm{vertices}) \ar[r]^-{\delta_c^0} & \oplus F(\mathrm{edges}) \ar[r]^-{\delta_c^1} & \oplus F(\mathrm{faces}) \ar[r] & \cdots }
\]
is defined to be the \textbf{compactly supported cohomology} of $F$, written $H^n_c(X;F)$.
\end{defn}

\begin{lem}
 $(C^{\bullet}_c(X;F),\delta^{\bullet}_c)$ is a chain complex.
\end{lem}
\begin{proof}
The result is a consequence of Definition~\ref{defn:signed-reln} along with linearity of the maps $\rho_{\tau,\sigma}$.
\end{proof}

To define the cellular sheaf cohomology of $F$, we simply remove all the cells from $X$ without compact closures and apply the same formula.

\begin{defn}[Ordinary Cohomology]\label{defn:ord-cohom}
    Let $X$ be a cell complex and $F:X\to\Vect$ a cellular sheaf. Let $j:X'\to X$ be the subcomplex consisting of cells that do not have vertices in the one-point compactification of $X$. Define the ordinary cochains and cohomology by
\[
  C^{\bullet}(X;F):=C^{\bullet}_c(X';j^*F) \qquad H^i(X;F):=H^i_c(X';j^*F)
\]
\end{defn}

\begin{rmk}
We say that Definition~\ref{defn:ord-cohom} is the ``ordinary cohomology'' because it is natural from the perspective of limits and injective resolutions, as we'll see in Section~\ref{sec:derived_sheaf_cohom}.
\end{rmk}

\begin{ex}[Compactly Supported vs. Ordinary Cohomology]
Consider the example of the half-open interval $X=[0,1)$ decomposed as $x=\{0\}$ and $a=(0,1)$. 
Now consider the constant sheaf $\Bbbk_X$. 
To compute compactly supported cohomology, we must first pick a local orientation of our space. By choosing the orientation that points to the right, we get that $[x:a]=-1$. 
The cohomology of our sheaf is computed via the complex
\[
  \xymatrix{\Bbbk \ar[r]^{-1} & \Bbbk},
\]
which yields $H^0_c=H^1_c=0$. 
If we follow the prescription for computing ordinary cellular sheaf cohomology, then we must remove the vector space sitting over $a$ in our computation. 
The resulting complex is simply the vector space $\Bbbk$ placed in degree 0, so $H^0(X;\Bbbk_X)=\Bbbk$ and all higher cohomologies vanish.
\end{ex}

\subsection{Cellular Cosheaf Homology}
\label{subsec:cell-cosheaf-homology}

For cellular cosheaves the exact dual construction works, but we use the term ``Borel-Moore'' in place of ``compactly supported.''
This is in part to honor Borel and Moore, whose landmark paper~\cite{borel1960homology} implicitly contains the cosheaf axiom.

\begin{defn}[Borel-Moore Cosheaf Homology]\label{defn:BM-hom}
 Let $X$ be a cell complex and let $\cosheaf{F}:X^{\op}\to\Vect$ be a cellular cosheaf. 
 Define the \textbf{Borel-Moore homology of} $H_{\bullet}^{BM}(X;\cosheaf{F})$ to be the homology of the complex $C^{BM}_{\bullet}(X;\cosheaf{F})$:
\[
 \xymatrix{ \cdots \ar[r] & \oplus \cosheaf{F}(\mathrm{faces}) \ar[r]^-{\partial_2} & \oplus \cosheaf{F}(\mathrm{edges}) \ar[r]^-{\partial_1} & \oplus \cosheaf{F}(\mathrm{vertices}) \ar[r] & 0 }
\]
\end{defn}

\begin{defn}[Ordinary Cosheaf Homology]
  Let $X$ be a cell complex and let $\cosheaf{F}:X^{\op}\to\Vect$ be a cellular cosheaf. By discarding all the cells without compact closure, we obtain the maximal compact subcomplex $X'$. If we write $j:X'\hookrightarrow X$ for the inclusion, then we can define the ordinary chain complex to be
  \[
  C_{\bullet}(X;\cosheaf{F})=C^{BM}_{\bullet}(X';j^*\cosheaf{F}).
  \]
  Applying Definition~\ref{defn:BM-hom} then yields \textbf{cosheaf homology} $H_{\bullet}(X;\cosheaf{F})$ of $\cosheaf{F}$.
\end{defn}

\section{The Derived Category of Cellular (Co)Sheaves}
\label{sec:derived_sheaf_cohom}

There is a surprising symmetry in the land of cellular sheaves and cosheaves.
The category of sheaves over a topological space $X$ only has, in general, enough injective objects, i.e.~every sheaf admits an injective resolution.
In Section~\ref{subsec:no-proj}, a topological criterion is given to determine when this is the case.
However, since the categories of cellular sheaves and cosheaves are equivalent to the functor categories $[X,\Vect]$ and $[X^{\op},\Vect]$, we have both enough injectives and projectives.
Recall that, when there are enough injectives and projectives, the bounded derived category of an abelian category can be described as the homotopy category of chain complexes of injective or projective objects~\cite[Thm 6.7]{aluffi}.
In Section~\ref{subsec:elem-inj-proj} we describe a simple class of injective and projective objects---they are constant (co)sheaves supported on the closure and open star of a cell.
In Section~\ref{subsec:suff-elem} we prove that every injective or projective object decomposes into a direct sum of these.

\subsection{No Projective Sheaves in General}
\label{subsec:no-proj}

Although the existence of enough injective sheaves is largely due to the algebraic properties of the data category $\dat$, the existence of enough projectives is driven by geometry of the underlying space.

\begin{prop}[\cite{MO-proj-sheaves}]\label{prop:no-proj}
 Suppose $X$ is a topological space with the property that there is a point $x\in X$ such that for every open neighborhood $U\ni x$ there is a strictly smaller open neighborhood $V\subset U$. Then the category of sheaves on $X$ does not have enough projectives.
\end{prop}
\begin{proof}
 Consider the map $i:x\hookrightarrow X$ and the skyscraper sheaf $i_*k$. 
 Suppose there is a projective sheaf $P$ and a surjection $P\to i_*k$. 
In the category of sheaves, this equivalent to the map on stalks $P_y\to (i_*k)_y$ being a surjection for every point $y\in X$.
This map is non-zero only when $y=x$, which implies that there is an open set $U\ni x$ such that $P(U)\to i_* k(U)$ is non-zero.
By our hypothesis there is another open set $V\subsetneq U$.
Consider the constant sheaf with value $k$ extended by zero on $V$, which we denote by $j_!k_V$. 
Note that we have the following diagram of sheaves
\[
 \xymatrix{j_!k_V \ar[r] & i_*k \ar[r] & 0\\ & P \ar@{.>}[lu] \ar[u] &}
\]
whose value on the open set $U$ is
\[
 \xymatrix{j_!\tilde{k}_V(U)=0 \ar[r] & i_*k(U)=k \ar[r] & 0\\ & P(U) \ar@{.>}[lu] \ar[u] &}
\]
so in particular the surjection must factor through zero, which contradicts the assumption that $P(U)\to i_* k(U)$ is non-zero.
\end{proof}

\subsection{Elementary Injectives and Projectives}
\label{subsec:elem-inj-proj}
 
\begin{defn}
Let $i_{\sigma}:\star \to X$ be the poset map that takes the unique element $\star$ to the cell $\sigma$ in $X$.
A cellular sheaf is an \define{elementary injective sheaf} if it is of the form $(i_{\sigma})_*W=:[\sigma]^W$ for some vector space $W$.
A cellular cosheaf is an \define{elementary projective cosheaf} if it is of the form $(i_{\sigma})_*\cosheaf{W}=:[\widehat{\sigma}]^W$ for some vector space $W$.
When the vector space $W$ is not specifified, it is assumed to be the ground field $\Bbbk$.
\end{defn}

\begin{rmk}[Supported on the Closure]
In order to avoid tracing through the definition of the appropriate Kan extension, we give a simple formula that tells you the value of these (co)sheaves on a cell $\tau$:
\begin{equation*}
[\sigma]^W(\tau)=[\widehat{\sigma}]^W(\tau)=
\begin{cases}
W & \text{if } \tau\leq \sigma,\\
0 & \text{other wise.}
\end{cases}
\end{equation*}
In other words, these (co)sheaves are supported on the closure $\bar{\sigma}$ of a cell $\sigma$.
\end{rmk}

\begin{lem}
The sheaf $[\sigma]^W=(i_{\sigma})_*W$ is injective and the cosheaf $[\cosheaf{\sigma}]^W=(i_{\sigma})_*\cosheaf{W}$ is projective.
\end{lem}
\begin{proof}
In order to check that the sheaf $[\sigma]^W=(i_{\sigma})_*W$ is injective, it suffices to chase the appropriate adjunction. To wit, if $\iota: \sheaf{A}\to \sheaf{B}$ is a monomorphism and $\eta:\sheaf{I}\to (i_{\sigma})_*W$ is any morphism, the adjunction says that
\[
    \Hom_{\Shv(X)}(A,(i_{\sigma})_*W)\cong \Hom_{\Vect} (A(\sigma),W)
\]
and hence the sheaf map $\eta$ is determined by the vector space map $\eta(\sigma):A(\sigma)\to W$.
Since $\iota:A\to B$ is a monomorphism in $[X,\Vect]$ the map $A(\sigma)\to B(\sigma)$ is a monomorphism. 
Fixing a basis for $A(\sigma)$ and extending it to a basis for $B(\sigma)$ allows us to define an extended map $\tilde{\eta}(\sigma):B(\sigma)\to W$, which by following the above adjunction in reverse gives a sheaf map $\tilde{\eta}:B\to (i_{\sigma})_*W$. A similar argument can be used to prove that $(i_{\sigma})_*\cosheaf{W}$ is projective.
\end{proof}

What is less commonly noted, and for good reason given Proposition~\ref{prop:no-proj}, are projective sheaves and injective cosheaves.

\begin{defn}
Let $i_{\sigma}:\star \to X$ be the map that assigns to the one element $\star$ the value $\sigma\in X$ and let $W$ be a vector space.
A cellular sheaf is an \define{elementary projective sheaf} if it is of the form $(i_{\sigma})_{\dd}W$.
Similarly, a cellular cosheaf is an \define{elementary injective cosheaf} if it is of the form $(i_{\sigma})_{\dd}\widehat{W}$. 
When the vector space $W$ is not specifified, it is assumed to be the ground field $\Bbbk$.
\end{defn}

\begin{rmk}[Supported on the Open Star]
Again, we provide a formula for what cells get what vector spaces.
\begin{equation*}
\{\sigma\}^W(\tau)=\{\widehat{\sigma}\}^W(\tau)=
\begin{cases}
W & \text{if } \sigma\leq \tau,\\
0 & \text{other wise.}
\end{cases}
\end{equation*}
In other words, these (co)sheaves are supported on the open star of $\sigma$.
\end{rmk}

We leave it as an exercise for the reader to check that the above sheaf is actually projective and the above cosheaf is actually injective.

\subsection{Sufficiency of Elementary Objects}
\label{subsec:suff-elem}

For a finite cell complex, the above described objects are sufficient for understanding all injective and projective objects. The following argument is a generalization of Shepard's~\cite[Thm.~1.3.2]{shepard}.

\begin{lem}\label{lem:inj}
Let $X$ be the face-relation poset of a finite cell complex. Every injective or projective sheaf (or cosheaf) is isomorphic to the direct sum of elementary injectives or elementary projectives, respectively.
\end{lem}
\begin{proof}
We prove that every injective sheaf is a direct sum of elementary injectives.
Assume for induction that every injective sheaf $I$ that is non-zero on at most $k\leq n-1$ cells is isomorphic to $\oplus_{\sigma}[\sigma]^{V_{\sigma}}$. 
Now consider an injective sheaf that is non-zero on exactly $n$ cells. 
Let $\sigma$ be a cell of maximal dimension where $I(\sigma)=:V\neq 0$. 
Since $I$ is zero on all higher cells incidence to $\sigma$, there is a non-zero map $\eta$ from the skyscraper sheaf $S^V_{\sigma}$ to $I$ with $\eta(\sigma)=\id_V$. 
There is also a non-zero map $\iota:S_{\sigma}^V\to[\sigma]^V$. 
This gives us a diagram
\[
 \xymatrix{0 \ar[r] & S_{\sigma}^V \ar[r]^{\iota} \ar[d]_{\eta} & [\sigma]^V \ar@{.>}[dl]^{\exists \tilde{\eta}} \\ & I &}
\]
and the implicated existence of a map $\tilde{\eta}:[\sigma]^V\to I$. 
If $\tau\leq \sigma$, then by the fact that $\tilde{\eta}$ is a sheaf map,
\[
\id_V=\tilde{\eta}(\sigma)\circ \rho^{[\sigma]}_{\sigma,\tau}=\rho^{I}_{\sigma,\tau}\circ \tilde{\eta}(\tau)
\]
the map $\tilde{\eta}$ is injective. 
Since every short exact sequence where the first object is injective splits, we can deduce that $I\cong [\sigma]^V\oplus \cok(\tilde{\eta})$. 
Since $\cok(\tilde{\eta})$ is zero wherever $I$ is and also zero on $\sigma$, it is non-zero on at most $n-1$ cells and the induction hypothesis applies. 
The zero sheaf is clearly equal to a direct sum of elementary injectives with the zero vector space, which checks the base case, completing the induction.

A similar argument applies to any projective sheaf $P$ where one instead considers the cell $\sigma$ of minimal dimension that $P$ is supported on. Repeating the argument with an elementary projective $\{\sigma\}^V\to S_{\sigma}^V$ provides a surjection $P\to \{\sigma\}^V$, which causes $P$ to split. Dualizing appropriately gives the analogous statements for cellular cosheaves.
\end{proof}







\section{The Derived Definitions of (Co)Homology}
\label{sec:derived-homology}

The standard definition of sheaf cohomology takes an injective resolution of $F$, applies $p_*$ to the injective resolution to obtain a complex of vector spaces and then takes cohomology of this complex.
For the face-relation poset of a cell complex, the categories $\Shv(X)$ and $\Coshv(X)$ have both enough injectives and projectives.
For a cellular sheaf, this allows us to either take an injective resolution and apply $p_*$ or to take a projective resolution and apply $p_{\dd}$ and a similar story holds for cosheaves.
This gives rise to four possible theories of interest---two for cell sheaves and two for cell cosheaves.
The two ``extra'' theories appear to be novel in the literature, although there are related ideas in Bredon~\cite{Bredon}.
It is an open question whether or not the theory of cellular sheaf homology presented here approximates Bredon's.
In Section~\ref{subsec:cosheaf-homology-colimit} we prove that taking a projective resolution of a cosheaf $\cosheaf{F}$ and applying $p_*$ gives a complex that is quasi-isomorphic to the complex of Section~\ref{subsec:cell-cosheaf-homology}.
In Section~\ref{subsubsec:subdivision} we use the derived definitions presented below to prove that these homology theories are invariant under subdivision.

\begin{defn}\label{defn:homology-theories}
Let $p:X\to\star$ denote the constant map. 
Let $I^{\bullet}$ be an injective and $P^{\bullet}$ a projective resolution of a cellular sheaf $\sheaf{F}$. 
Similarly, let $\cosheaf{I}^{\bullet}$ be an injective and $\cosheaf{P}^{\bullet}$ a projective resolution of a cellular cosheaf $\cosheaf{F}$. We define the
\begin{itemize}
\item[-] $i^{th}$ \define{cellular sheaf cohomology} $H^i(X;\sheaf{F})$ to be $H^i(p_* \sheaf{I}^{\bullet})$,
\item[-] $i^{th}$ \define{cellular cosheaf homology} $H_i(X;\cosheaf{F})$ to be $H_i(p_*\cosheaf{P}^{\bullet})$,
\item[-] $i^{th}$ \define{cellular sheaf homology} $H_i(X;\sheaf{F})$ to be $H_i(p_{\dd} \sheaf{P}^{\bullet})$, and
\item[-] the $i^{th}$ \define{cellular cosheaf cohomology} $H^i(X;\cosheaf{F})$ to be $H^i(p_{\dd} \cosheaf{I}^{\bullet})$.
\end{itemize}
\end{defn}

By using Poincar\'e duality, we can show these latter two theories have content by relating them to the expected ones.

\begin{prop}[Homological Poincar\'e Duality]
\label{prop:homological-poincare-duality}
Suppose $X$ is the poset of cells of a triangulation of a closed $n$-manifold. Suppose $\sheaf{F}$ is a cellular sheaf on $X$ and $\cosheaf{F}$ is the cellular cosheaf determined on the dual triangulation $\tilde{X}$ as defined in Proposition~\ref{thm:mfld_sheaf_cosheaf}, then
\[
H^i(X;\sheaf{F})=H^i(\tilde{X};\cosheaf{F}) \qquad \mathrm{and} \qquad H_i(X;\sheaf{F})=H_i(\tilde{X};\cosheaf{F}).
\]
\end{prop}
\begin{proof}
The proof is immediate from the fact that the two diagrams of vector spaces are the same.
In more detail, the injective sheaf $[\sigma]^W$ can also be viewed as the injective cosheaf $\{\widehat{\tilde{\sigma}}\}^W$ on the dual triangulation, hence an injective resolution of $\sheaf{F}$ can be viewed as an injective resolution of $\cosheaf{F}$ on the dual triangulation.
Similarly, the projective sheaf $\{\sigma\}^W$ can also be viewed as the projective cosheaf $[\widehat{\tilde{\sigma}}]^W$ over the dual triangulation.
\end{proof}

\subsection{Cosheaf Homology}
\label{subsec:cosheaf-homology-colimit}

In Sections~\ref{subsec:cell-sheaf-cohomology} and~\ref{subsec:cell-cosheaf-homology} cellular sheaf cohomology and cellular cosheaf homology were defined without using injective or projectives. Shepard outlines a proof that derived functors of the limit agree with the formula presented in~\ref{subsec:cell-sheaf-cohomology}.
Below, we prove that derived functors of the colimit agree with the formulas of~\ref{subsec:cell-cosheaf-homology}.

\begin{thm}
The left derived functors of $p_{*}$ agree with the computational formula for homology, i.e. $L_i p_{*}\widehat{F}=H_i(X;\widehat{F})$.
\end{thm}

\begin{proof}
Begin with a projective resolution of $\widehat{P}_{\bullet}\to\widehat{F}$ and then take cellular chains of each cosheaf to obtain the following double complex:
\[
 \xymatrix{ & \vdots & \vdots & \vdots & \\
\cdots &  C_1(X;\widehat{P}_1) \ar[r]\ar[d] & C_1(X;\widehat{P}_0) \ar[r]\ar[d] & C_1(X;\widehat{F}) \ar[r]\ar[d] & 0 \\
\cdots & C_0(X;\widehat{P}_1)\ar[r]\ar[d] & C_0(X;\widehat{P}_0)\ar[r]\ar[d] & C_0(X;\widehat{F})\ar[r]\ar[d] & 0 \\
\cdots & \colim \widehat{P}_1 \ar[r] \ar[d]& \colim \widehat{P}_0\ar[r]\ar[d] & \colim \widehat{F}\ar[r]\ar[d] & 0 \\
& 0 & 0 & 0 &}
\]

\begin{lem}
 For $\widehat{P}$ a projective cosheaf 
$$H_i(C_{\bullet}^{BM}(X;\widehat{P}))\cong H_i(C_{\bullet}(X;\widehat{P}))\cong 0$$ for $i>0$.
\end{lem}
\begin{proof}
 Observe that we can assume that $\widehat{P}$ is an elementary projective co-sheaf with value $k$, i.e. $[\widehat{\sigma}]$, since $C_{\bullet}^{BM}(X;\oplus A_i)=\oplus C_{\bullet}^{BM}(X;A_i)$.

Everything follows from the following consequence of our definition of a cell complex: In the one-point compactification of $X$, the closure of any cell $\sigma\in X$, call it $|\bar{\bar{\sigma}}|$, has the homeomorphism type of a closed $n$-simplex.

$C_{\bullet}(X;[\widehat{\sigma}])$ is the chain complex that computes the cellular homology of $Y=|\{\tau\leq \sigma|\bar{\tau}\,\mathrm{is}\,\mathrm{compact}\}|$, which is a closed $n$-simplex minus the star of a vertex. On the other hand, $C^{BM}_{\bullet}(X;[\widehat{\sigma}])$ is equal to the chain complex calculating the cellular homology of $|\bar{\bar{\sigma}}|$ except in degree zero if $|\bar{\sigma}|$ is not compact. Notice that $H_1$ for both of these complexes is the same, as $|\bar{\sigma}|$ and $|\bar{\bar{\sigma}}|$ are simply connected. This proves the claim.
\end{proof}

\begin{lem}
 For any cellular cosheaf $\widehat{F}$ on a cell complex $X$ we have that $$\colim \widehat{F}\cong\cok(C_1(X;\widehat{F})\to C_0(X;\widehat{F})).$$
\end{lem}
\begin{proof}
 First let us prove that taking the coproduct of $\widehat{F}$ over all the cells obtains a vector space that surjects onto the colimit. As part of the definition of $\colim \widehat{F}$ is a choice of maps $\psi_{\sigma}:\widehat{F}(\sigma)\to\colim \widehat{F}$.  Let $\Psi=\oplus \psi_{\sigma}:\oplus \widehat{F}(\sigma)\to \colim \widehat{F}$, now consider the factorization of this map through the image:
\[
 \xymatrix{\oplus \widehat{F}(\sigma) \ar[rr]^{\Psi} \ar[rd] & & \colim \widehat{F} \\
& \im \Psi \ar[ru]^j &}
\]
Now we can use the $\im \Psi$ to define a new co-cone over the diagram $\widehat{F}$ simply by pre-composing the factorized map with the inclusions $i_{\sigma}:\widehat{F}(\sigma)\to\oplus \widehat{F}(\sigma)$. Since the colimit is the initial object in the category of co-cones, there must be a map $u:\colim \widehat{F}\to \im \Psi$ and thus $u\circ j=\id$ since there is only one map $\colim\widehat{F}\to\colim\widehat{F}$.

Now observe that $C_0(X;\widehat{F})=\oplus \widehat{F}(v_i)$ surjects onto the colimit of $\widehat{F}$ by virtue of the fact that since every cell $\sigma\in X$ has at least one vertex as a face, the map $\Psi$ factors through $\oplus \widehat{F}(v_i)$. Thus there is a surjection from $\Psi':C_0(X;\widehat{F})\to \colim\widehat{F}$. Notice that by universal properties of the cokernel of $\partial_0:C_1(X;\widehat{F})\to C_0(X;\widehat{F})$ it suffices to check that $\Psi'\circ\partial_0=0$. However, this is clear since every $e$ edge has two vertices $v_1$ and $v_2$ (we've discarded all those edges without compact closures), then we need only check the claim for each diagram of the form
\[
 \xymatrix{ & \widehat{F}(e) \ar[ld]_{r_{e,v_1}} \ar[rd]^{r_{e,v_2}} & \\
\widehat{F}(v_1) & & \widehat{F}(v_2)}
\]
where it is clear that the colimit can be written as $\widehat{F}(v_1)\oplus \widehat{F}(v_2)$ modulo the equivalence relation $(r_{e,v_1}(w),0)\simeq (0,r_{e,v_2}(w))$, i.e. 
$\partial_0|_e(w)=(-r_{e,v_1}(w),r_{e,v_2}(w))\simeq(0,0)$.
\end{proof}

From these two theorems we can conclude that the columns away from the chain complex of $\widehat{F}$ are exact and thus $\mathrm{Tot}_{\bullet}(C_i(X;\widehat{P}_j))$ induces quasi-isomorphisms between $\colim \widehat{P}_{\bullet}$ and $C_{\bullet}(X;\widehat{F})$. We have thus established the theorem.
\end{proof}

\subsection{Borel-Moore Cosheaf Homology}
\label{subsubsec:BM_cosheaf_homology}

\begin{defn}\label{defn:BM_cosheaf_homology}
 Suppose $\cosheaf{F}$ is a cellular cosheaf. Define $\Gamma^{BM}(X;\cosheaf{F})$ to be the colimit of the diagram extended over the one-point compactification of $X$ where we define $\cosheaf{F}(\infty)=0$.
\end{defn}

Now we can prove that the formula provided calculates the Borel-Moore homology of a cosheaf $\widehat{F}$ by establishing the following lemma:

\begin{lem}\label{lem:BM_cosheaf_homology}
 For any cellular cosheaf $\widehat{F}$ on a cell complex $X$ we have that $\Gamma^{BM}(X;\widehat{F})\cong \cok(C_1^{BM}(X;\widehat{F})\to C^{BM}_0(X;\widehat{F}))$.
\end{lem}
\begin{proof}
 The proof above goes through until the last argument. Now we have edges $e$ with only one vertex. However, by extending and zeroing out at infinity to get that the colimit of
\[
 \xymatrix{ & \widehat{F}(e) \ar[ld]_{r_{e,v}} \ar[rd]^0 & \\
\widehat{F}(v) & & \widehat{F}(\infty)=0}
\]
is exactly equal to the co-equalizer of $r_{e,v}:\widehat{F}(e)\to F(v)$ and the zero morphism, i.e. the cokernel.
\end{proof}

\subsection{Invariance under Subdivision}
\label{subsubsec:subdivision}

One of the advantanges of using the derived perspective is that it is easy to prove that sheaf cohomology and cosheaf homology are invariant under subdivision by using only categorical methods.

\begin{defn}[\cite{shepard} 1.5, p.29]\index{subdivision}\index{cell complex!subdivision}
 A \textbf{subdivision} of a cell complex $X$ is a cell complex $X'$ with $|X'|=|X|$ and where every cell of $X$ is a union of cells of $X'$.
\end{defn}

Untangling the definition a bit we see that if $\sigma$ is a cell of $X$, then there is a collection of cells $\{\sigma_i'\}$ such that $\cup_i |\sigma_i'|=|\sigma|$. As such, we can define a surjective map of posets $s:X'\to X$ defined by making $s(\sigma')=\sigma$ if $|\sigma'|\subseteq |\sigma|$.

\begin{prop}
 Subdivision of a cell complex $X$ induces an order preserving map $s:X'\to X$ of the corresponding face-relation posets.
\end{prop}
\begin{proof}
 The ordering on $X'$ is given by the face relation. Suppose $\sigma'\leq \tau'$, then either $s(\sigma')=s(\tau')$ or not. If not, then $\sigma'$ and $\tau'$ belong to the subdivision of two cells $\sigma\leq \tau$.
\end{proof}

We are going to use this fact to define the subdivision of a sheaf in a cleaner manner than is found in~\cite{shepard}.

\begin{defn}
Let $\sheaf{F}$ or $\cosheaf{F}$ be a sheaf or cosheaf on $X$. For a subdivision $s:X'\to X$ of $X$ we get a \define{subdivided sheaf} $F':=s^*F$ or \define{subdivided cosheaf} $\cosheaf{F}':=s^*\cosheaf{F}$.
\end{defn}

\begin{ex}
Let $X$ be the unit interval $[0,1]$ stratified with $x=0$, $y=1$ and $a=(0,1)$. 
Consider the barycentric subdivision of $X$, which produces a third vertex $\bar{a}$ and two edges $a_x$ and $a_y$. 
Given a cellular sheaf $F$, the subdivided sheaf has $F'(\bar{a})=F'(a_x)=F'(a_y)=F(a)$ where we use the identity map for the two new restriction maps. 
Observe that if $F$ is the elementary injective sheaf $[a]$, then $F'$ is \emph{not} an injective sheaf, yet nevertheless $F'$ and $F$ have isomorphic cohomology.
\end{ex}

To understand the behavior of any of the homology theories associated to cellular sheaves or cosheaves, one need only consider the following diagram of posets
\[
 \xymatrix{X' \ar[rr]^s \ar[rd]_{p_{X'}} & & X\ar[ld]^{p_X}\\
        &\star&}
\]
and the associated pushforward functors. 
For instance, we can prove that cellular sheaf cohomology is invariant under subdivision in a different way than Shepard, cf.~\cite[Thm. 1.5.2]{shepard}.

\begin{thm}\label{thm:subdivision}
 Suppose $F$ is a sheaf on $X$ and $X'$ is a subdivision of $X$, then
\[
 H^{\bullet}(X;F)\cong H^{\bullet}(X';F')
\]
The analogous result holds for cosheaves.
\end{thm}
\begin{proof}
Observe that since $p_{X'}=p_X\circ s$, then $(p_{X'})_*=(p_X)_*\circ s_*$. Now recall $$(p_{X'})_*F'=(p_{X'})_*s^*F=(p_{X})_*\circ s_* s^* F.$$ 
The question then boils down to understanding the relationship between $s_*s^*F$ and $F$. Unraveling the definition reveals
\begin{eqnarray*}
 s_*s^*F(y) & = &\varprojlim \{s^*F(x) | s(x)\geq y\} \\
&=&\varprojlim \{F(s(x)) | s(x)\geq y\} \\
(\mathrm{surjectivity})&=&\varprojlim \{F(x) | x\geq y\} \\
(\mathrm{sheaf-axiom})&=& F(U_y) \\
&=& F(y)
\end{eqnarray*}
So we have that for the subdivision map $s_*s^*F\cong F$ and as a consequence
$$(p_{X'})F'\cong p_X F.$$
Now we can just take the associated right derived functors to obtain the result.
\end{proof}

\begin{rmk}
Following a proof along similar lines, one can show that all of the four theories introduced are invariant under subdivision~\cite{curry2014sheaves}.
\end{rmk}

\section{Derived Equivalence of Sheaves and Cosheaves}
\label{sec:derived-equivalence}

In the Spring of 2012, Robert MacPherson conjectured that the category of cellular sheaves and the category of cellular cosheaves are derived equivalent.
A few weeks later the author provided a proof.
Subsequently, the author learned that a similar result was obtained by Peter Schneider in~\cite{schneider-vd} who used the term ``coefficient system'' instead of cellular cosheaf.
However, the proof of the equivalence (Theorem~\ref{thm:equivalence}) given here is fundamentally different than Schneider's because of its use of elementary injective sheaves and skyscraper cosheaves.
 
As noted in the introduction, this equivalence associates to a sheaf $F$ the cosheaf of compactly supported cochains valued in $F$:
\[
\equivp(F): \qquad U \qquad \squigrightarrow \qquad C^{\bullet}_c (U;F)
\]
This is the perspective given in Definition~\ref{defn:equiv-fun-cptsupp}.
Alternatively, this equivalence associates to a cellular sheaf $F$ a complex of elementary projective cosheaves.
This is the perspective given in Definition~\ref{defn:equiv-fun-projectives}.
By applying linear duality, one can view these formulas as giving exactly the formula for Verdier duality discovered by Shepard~\cite{shepard}, who never published his result.
Finally, in Section~\ref{subsec:classical-duality}, we conclude by proving some original duality statements that use the theory of sheaf homology presented here.

\subsection{Definition in Terms of Compactly Supported Cochains}

In this section we describe a formula that provides an explicit derived equivalence between the category of cellular sheaves and the category of cellular cosheaves. Recall that if $X$ is the face-relation poset of a cell complex, then for every cell $\sigma$ the open star $U_{\sigma}=\{\tau\,|\, \sigma\leq\tau\}$ is the set of cells that are cofaces of $\sigma$. Moreover, this association is order-reversing, i.e.
\[
\lambda \leq \sigma \qquad \rightsquigarrow \qquad U_{\sigma} \subseteq U_{\lambda}.
\]

\begin{defn}\label{defn:equiv-fun-cptsupp}
For a cellular sheaf $\sheaf{F}:X\to\Vect$ we define the associated \define{dual complex} of cellular cosheaves to be the following cosheaf of chain complexes
\[
\equivp(\sheaf{F})(\sigma) := C^{\bullet}_c(U_{\sigma};F)= F(\sigma) \to \bigoplus_{\sigma\leq_1\tau} F(\tau) \to \bigoplus_{\sigma\leq_2\gamma} F(\gamma) \to \cdots
\]
where the complex is shifted so that $F(\sigma)$ is placed in cohomological degree $\dim|\sigma|$ or, equivalently, homological degree $-\dim|\sigma|$. 
If $\lambda\leq\sigma$, then the associated map of complexes is the natural inclusion:
\[
r^{\bullet}_{\lambda,\sigma}: C^{\bullet}_c(U_{\sigma};F) \to C^{\bullet}_c(U_{\lambda};F)
\] 
\end{defn}

\begin{rmk}
To better understand the extension maps of $\equivp(F)$, consider the case where $\lambda$ is a codimension one face of $\sigma$:
\[
  \xymatrix{0 \ar[r] \ar[d]_{r^{i-1}_{\lambda,\tau}} & F(\sigma) \ar[r]^-{\delta_c^i} \ar[d]_{r^{i}_{\lambda,\tau}} & \bigoplus\limits_{\sigma\leq_1\tau} F(\tau) \ar[r]^-{\delta_c^{i+1}} \ar[d]_{r^{i+1}_{\lambda,\tau}} & \bigoplus\limits_{\sigma\leq_2\gamma} F(\gamma) \ar[d]_{r^{i+2}_{\lambda,\tau}} \\
  F(\lambda) \ar[r]_-{\delta_c^{i-1}} & \bigoplus\limits_{\lambda\leq_1\sigma} F(\sigma) \ar[r]_-{\delta_c^i} & \bigoplus\limits_{\lambda\leq_2\tau} F(\tau) \ar[r]_-{\delta_c^{i+1}} & \bigoplus\limits_{\lambda\leq_3\gamma} F(\gamma) }.
\]
\end{rmk}

We illustrate this definition in a simple example.
 
\begin{ex}[Closed Interval]
Suppose we start with a cellular sheaf $F$ on the unit interval $|X|=[0,1]$ decomposed into cells $x=0$, $y=1$, and $a=(0,1)$. Such a cellular sheaf is just a diagram of vector spaces of the following form:
\[
  \xymatrix{& F(a) & \\ F(x) \ar[ur]^{\rho_{a,x}} & & \ar[lu]_{\rho_{a,y}} F(y)}
\]
The associated dual can be visualized as follows:
\[
\xymatrix{F(a) & \ar[l]_{\id} F(a) \ar[r]^{\id} & F(a) \\ F(x) \ar[u]^{\rho_{a,x}} & \ar[l] 0 \ar[u] \ar[r] & \ar[u]_{\rho_{a,y}} F(y) }
\]
The horizontal arrows are now rightly viewed as the extension maps of a cellular cosheaf of chain complexes:
\[
  \xymatrix{& \ar[ld]_{r^{\bullet}_{x,a}} \equivp(F)(a) \ar[rd]^{r^{\bullet}_{y,a}} & \\ \equivp(F)(x) & & \equivp(F)(y) }
\]
\end{ex}

\begin{prop}[Local Homology, cf.~\cite{schneider-vd}]
\label{prop:local-homology}
Let $X$ be a cell complex.
If we define $\partial U_{\sigma}$ to be set of cells in the closure of $U_{\sigma}$, but not in $U_{\sigma}$ itself, then the complex $\equivp(F)(\sigma)$ is the formula for computing relative cohomology of the pair, i.e.
\[
\equivp(F)(\sigma) \simeq H^{\ast}(U_{\sigma},\partial U_{\sigma};F)\cong H^{\ast}(U_{\sigma}/\partial U_{\sigma};F)
\]
\end{prop}
\begin{proof}
One can check directly and note that when necessary, e.g.~if $\sigma$ is not a vertex, the open star can be subdivided further to get a cell complex without affecting the answer, cf.~Theorem \ref{thm:subdivision}.
\end{proof}

\begin{ex}[Manifold]
Let  $\Bbbk_M$ be the constant sheaf on a cell complex $M$ that happens to be a manifold of dimension $n$ without boundary. Since $U_{\sigma}$ is homeomorphic to $\RR^n$, $\equivp(\Bbbk_M)(\sigma)$ is the complex that computes compactly supported cohomology of $\RR^n$. 
Alternatively, in view of Proposition~\ref{prop:local-homology}, the complex computes cohomology of the sphere, which when shifted to homological indexing gives the following quasi-isomorphism:
\[
\equivp(\Bbbk_M)(\sigma)\simeq \Bbbk[-n]
\]
\end{ex}

\begin{rmk}
The complex of cosheaves $\equivp(\Bbbk_X)$ can be used to identify singularities in a space. 
For example, one could study the complex when $X$ is a graph with nodes of valence three or higher.
\end{rmk}

\subsection{Definition in terms of Elementary Objects}

Now we recast this definition so that it is a functor that takes in a cellular sheaf and returns a complex of elementary projective cosheaves.
This will streamline the proof that this functor is indeed a derived equivalence.

\begin{defn}\label{defn:equiv-fun-projectives}
  Let $D^b(\Shv(X))$ and $D^b(\Coshv(X))$ denote the derived categories of cellular sheaves and cosheaves respectively. The functor 
  $$\equivp:D^b(\Shv(X))\to D^b(\Coshv(X))$$ assigns to a sheaf $F\in \Shv(X)$ the following complex of projective cosheaves
  \[
    \xymatrix{\cdots \ar[r] & \bigoplus\limits_{\sigma\in X^i} [\widehat{\sigma}]^{F(\sigma)} \ar[r] & \bigoplus\limits_{\gamma\in X^{i+1}}[\widehat{\gamma}]^{F(\gamma)} \ar[r] & \bigoplus\limits_{\tau\in X^{i+2}}[\widehat{\tau}]^{F(\tau)} \ar[r] & \cdots}
  \]
  where the cohomological degree (or negative homological degree) of each term corresponds to dimension of the cell.
  The maps in between are to be understood as the matrix $\oplus [\sigma:\gamma]\rho^F_{\sigma,\gamma}$.
For a complex of sheaves, we simply form a double complex by first applying the above formula to each sheaf and then pass to the totalization.
\end{defn}

\begin{prop}
Definitions~\ref{defn:equiv-fun-cptsupp} and~\ref{defn:equiv-fun-projectives} are equivalent.
\end{prop}

\begin{proof}
By evaluating the formula given in Definition~\ref{defn:equiv-fun-projectives} on a cell and using the fact that each elementary projective is supported on the closure of the cell, one obtains the formula given in Definition~\ref{defn:equiv-fun-cptsupp}.
\end{proof}

\subsection{The Equivalence}

\begin{thm}[Equivalence]\label{thm:equivalence}
  $\equivp:D^b(\Shv(X))\to D^b(\Coshv(X))$ is an equivalence.
\end{thm}
\begin{proof}
First let us point out that the functor $\equivp$ really is a functor. Indeed if $\alpha:F\to G$ is a map of sheaves then we have maps $\alpha(\sigma):F(\sigma)\to G(\sigma)$ that commute with the respective restriction maps $\rho^F$ and $\rho^G$. As a result, we get maps $[\widehat{\sigma}]^{F(\sigma)}\to [\widehat{\sigma}]^{G(\sigma)}$. Moreover, these maps respect the differentials in $\equivp(F)$ and $\equivp(G)$, so we get a chain map. It is clearly additive, i.e. for maps $\alpha,\beta:F\to G$ $\equivp(\alpha+\beta)=P(\alpha)+P(\beta)$. This implies that $\equivp$ preserves homotopies.

It is also clear that $\equivp$ preserves quasi-isomorphisms. Note that a sequence of cellular sheaves $A^{\bullet}$ is exact if and only if $A^{\bullet}(\sigma)$ is an exact sequence of vector spaces for every $\sigma\in X$. This implies that $\equivp(A^{\bullet})$ is a double-complex with exact rows. By the acyclic assembly lemma~\cite[Lem.~2.7.3]{weibel} we get that the totalization is exact.

Since the derived category is built out of elementary injective sheaves, let us understand what this functor does to an elementary the injective sheaf $[\sigma]^V$. 
Applying formula we get
\[
\xymatrix{
 \equivp: [\sigma]^V &\rightsquigarrow& \cdots \ar[r] & \bigoplus\limits_{\tau^i\subset\sigma}[\widehat{\tau}]^V \ar[r] &\cdots \ar[r] & [\widehat{\sigma}]^V,
}
\]
which is nothing other than the projective cosheaf resolution of the skyscraper (or stalk) cosheaf $\widehat{S}_{\sigma}^V$ supported on $\sigma$, i.e.
\[
  \widehat{S}_\sigma^V(\tau)=\left\{\begin{array}{ll} V &\sigma=\tau\\ 0 & \mathrm{o.w.}\end{array}\right.
\]
Consequently, there is a quasi-isomorphism $q:\equivp([\sigma]^V)\to \widehat{S}_{\sigma}^V[-\dim\sigma]$ where $\widehat{S}_{\sigma}^V$ is placed in degree equal to the dimension of $\sigma$ assuming that $[\sigma]^V$ is initially in degree 0. 
By using an analogous definition and letting $\equivpc$ send cosheaves to sheaves, we see that
\[
 \equivpc(q):\equivpc\equivp([\sigma]^V)\to \equivpc(S_{\sigma}^V)=[\sigma]^V
\]
and thus we can define partially a natural transformation from $\equivpc\equivp$ to $\id_{D^b(\Shv)}$ when restricted to elementary injectives. 
However, by Lemma~\ref{lem:inj} we know that every injective looks like a sum of elementary injectives, so this natural transformation is well-defined for injective sheaves concentrated in a single degree. 
Moreover, $\equivp$ sends a complex of injectives, before taking the totalization of the double complex, to the projective resolutions of a complex of skyscraper cosheaves. 
Applying $\equivp$ to the quasi-isomorphism relating the double complex of projective cosheaves to the complex of skyscrapers, extends the natural transformation to the whole derived category. However, since $\equivp$ preserves quasi-isomorphisms, this natural transformation is in fact an equivalence. 
This shows $\equivpc\equivp\cong\id_{D^b(\Shv)}$. 
Repeating the argument starting from cosheaves shows that
\[
 \equivp:D^b(\Shv(X))\leftrightarrow D^b(\Coshv(X)):\equivpc
\]
is an adjoint equivalence of categories.
\end{proof}

\subsection{Verdier Duality}

Theorem~\ref{thm:equivalence} should be taken as the primary duality result from which other dualities spring.
Verdier duality for cellular sheaves simply takes the equivalence described above and applies linear duality, as we now describe.

\begin{defn}[Dualizing Complex]
Let $X$ be a cell complex, the \define{dualizing complex}
$\omega_X^{\bullet}$ is a complex of elementary injective sheaves that in negative cohomological degree $i$ is sum over the elementary injectives concentrated on $i$-cells, i.e.
\[
\omega^{-i}_X:=\bigoplus\limits_{\gamma\in X^i}[\gamma].
\]
The maps between these terms use the orientations on cells to guarantee it is a complex, i.e.
\[
\xymatrix{\cdots \ar[r] & \bigoplus\limits_{\tau\in X^{i+1}}[\tau] \ar[r]^-{\oplus[\gamma:\tau]} & \bigoplus\limits_{\gamma\in X^i} [\gamma] \ar[r]^-{\oplus[\sigma:\gamma]} & \bigoplus\limits_{\sigma\in X^{i-1}} [\sigma] \ar[r] &\cdots }
\]
\end{defn}

\begin{defn}[Sheaf Hom]
Given two cellular sheaves $F$ and $G$, we get a third cellular sheaf $\mathcal{H}om(F,G)$ called \define{sheaf hom} whose value on a cell $\sigma$ is given by the space of natural transformations between the restrictions of $F$ and $G$ to the star of $\sigma$, i.e.
$$\mathcal{H}om(F,G)(\sigma):=\Hom(F|_{U_{\sigma}},G|_{U_{\sigma}}).$$
\end{defn}

Let $\Shv_f(X)$ denote the category of cellular sheaves valued in \emph{finite} dimensional vector spaces $\vect$.
For $\vect$ there is an obvious anti-autoequivalence that takes
a vector space $V$ to its linear dual $V^*$
and a map of vector spaces $\rho:V\to W$ to its adjoint $\rho^*:W^*\to V^*$.

\begin{defn}[Verdier Dual]\label{defn:verdier-dual}
We define the \define{Verdier dual functor} as follows:
\[
D:D(\Shv_f(X))\to D(\Shv_f(X))^{\op} \qquad F \squigrightarrow D^{\bullet}F:=\mathcal{H}om(F,\omega^{\bullet}_X)
\]
Written out explicitly, this associates to $F$ the following complex of injective sheaves:
\[
\xymatrix{\cdots \ar[r] & \bigoplus\limits_{\tau\in X^{i+1}} [\tau]^{F(\tau)^*} \ar[r]_-{\oplus[\gamma:\tau]\rho^*} & \bigoplus\limits_{\gamma\in X^i} [\gamma]^{F(\gamma)^*} \ar[r]_-{\oplus[\sigma:\gamma]\rho^*} & \bigoplus\limits_{\sigma\in X^{i-1}}[\sigma]^{F(\sigma)^*} \ar[r] &\cdots }
\]
\end{defn}

Linear duality thus takes a cellular sheaf $F$ to a cellular cosheaf $\cosheaf{F}$ and vice versa since the restriction maps get dualized into extension maps. 
Moreover, this functor is contravariant since a sheaf morphism $F\to G$ gets sent to a cosheaf morphism in the opposite direction $\cosheaf{F}\leftarrow \cosheaf{G}$ as one can easily check. 
We can promote this functor to the derived category by exchanging cohomological indexing with homological indexing with appropriate shifts.
Let $V:D^b(\Coshv_f(X))\to D^b(\Shv_f(X))^{\op}$ denote this operation.
The following proposition is immediate and makes clear that by thinking of Verdier duality as an exchange of sheaves and cosheaves, then one can avoid the linear duals in Definition~\ref{defn:verdier-dual}.

\begin{prop}
  The functor $\equivp:D^b(\Shv_f(X))\to D^b(\Coshv_f(X))$ composed with linear duality $V:D^b(\Coshv_f(X))\to D^b(\Shv_f(X))^{\op}$ gives the Verdier dual anti-equivalence, i.e. $D\cong V\equivp$.
\end{prop}

\subsection{Homological Duality Statements}
\label{subsec:classical-duality}

In this section, we prove some homological duality results that appear classical at first glance. 

\begin{thm}
  If $F$ is a cell sheaf on a cell complex $X$, then
  \[
    H^i_c(X;F)\cong H_{-i}(X;\equivp(F)).
  \]
\end{thm}
\begin{proof}
First we apply the equivalence functor $\equivp$ to $F$
\[
  \xymatrix{0 \ar[r] & \bigoplus\limits_{v\in X}[\widehat{v}]^{F(v)} \ar[r] & \bigoplus\limits_{e\in X}[\widehat{e}]^{F(e)} \ar[r] & \bigoplus\limits_{\sigma\in X}[\widehat{\sigma}]^{F(\sigma)} \ar[r] & \cdots}
\]
Taking colimits (pushing forward to a point) term by term produces the complex of vector spaces
\[
  \xymatrix{0 \ar[r] & \bigoplus\limits_{v\in X}F(v) \ar[r] & \bigoplus\limits_{e\in X}F(e) \ar[r] & \bigoplus\limits_{\sigma\in X} F(\sigma) \ar[r] & \cdots}
\]
which the reader should recognize as being the computational formula for computing compactly supported sheaf cohomology.
\end{proof}

Now let us give a simple proof of the standard Poincar\'e duality statement on a manifold $X$ with coefficients in an arbitrary cell sheaf $F$, except this time the sheaf homology groups are used. 

\begin{thm}\label{thm:mfld_duality}\index{duality!over a manifold}
  Suppose $F$ is a cell sheaf on a cell complex $X$ that happens to be a compact manifold (so it has a dual cell structure $\widehat{X}$), then
  \[
    H^i(X;F)\cong H_{n-i}(X;F).
  \]
  Where the group on the right is not just notational, but it indicates the left-derived functors of $p_{\dd}$ on sheaves.
\end{thm}

\begin{proof}
  We repeat the first step of the proof of the previous theorem.
  By applying $\equivp$ to $F$ we get a complex of cosheaves. 
  Pushing forward to a point yields a complex whose (co)homology is the compactly supported cohomology of the sheaf $F$. 
  Now we recognize the formula~\ref{defn:BM-hom} for the Borel-Moore homology of a cosheaf defined on the dual cell structure.
  \[
    \xymatrix{0 \ar[r] & \ar@{~>}[d] \bigoplus\limits_{v\in X}F(v) \ar[r] & \ar@{~>}[d] \bigoplus\limits_{e\in X}F(e) \ar[r] & \ar@{~>}[d] \bigoplus\limits_{\sigma\in X} F(\sigma) \ar[r] & \cdots \\
    0 \ar[r] & \bigoplus\limits_{\tilde{v}\in \tilde{X}}\cosheaf{F}(\tilde{v}) \ar[r] & \bigoplus\limits_{\tilde{e}\in \tilde{X}}\cosheaf{F}(\tilde{e}) \ar[r] & \bigoplus\limits_{\tilde{\sigma}\in \tilde{X}}\cosheaf{F}(\tilde{\sigma}) \ar[r] & \cdots}
  \]
  Taking the homology of the bottom row is the usual formula for the Borel-Moore homology of a cellular cosheaf except the top dimensional cells are place in degree 0, the $n-1$ cells in degree -1, and so on. Everything being shifted by $n=\dim X$ we get the isomorphism
  \[
    H_{-i}(X;\equivp(F))\cong H^{BM}_{n-i}(\tilde{X};\cosheaf{F}).
  \]
  However, we already observed in Theorem~\ref{thm:mfld_sheaf_cosheaf} that the diagrams $\cosheaf{F}$ on $\tilde{X}$ and $F$ on $X$ are the same in every possible way, so in particular sheaf homology of $F$ must coincide with cosheaf homology of $\cosheaf{F}$. Thus using compactness to drop the Borel-Moore label and chaining together the previous theorem we get
  \[
    H^i(X;F)\cong H_{-i}(X;\equivp(F))\cong H_{n-i}(\tilde{X};\cosheaf{F})\cong H_{n-i}(X;F).
  \]
\end{proof}

\section{Compactly Supported Cohomology as a Coend}
\label{sec:coend}

We conclude with a broader discussion of why Verdier duality might be rightly viewed as an exchange of sheaves and cosheaves.
The inspiration for this discussion stems from the following result of Pitts~\cite{pitts}.

\begin{thm}[Pitts' Theorem]
  Let $X$ be a topological space. Every colimit-preserving functor on sheaves arises by forming the coend with a cosheaf, i.e. 
  $$\Coshv(X;\Set)\cong\Fun^{\text{cocts}}(\Shv(X;\Set),\Set).$$
\end{thm}

Although an enriched version of Pitts' theorem---one valid for sheaves and cosheaves of vector spaces---does not appear in the literature, the result should be valid for (co)sheaves valued in any monoidal category~$\dat$.
Let us assume for the moment that such a generalization holds.
Recall that the first formulation of Verdier duality was the statement that $Rf_!$ admits a right adjoint on the level of the derived category. 
This adjunction is sometimes called \define{global Verdier duality} and says that we have the following natural isomorphism:
\[
  \Hom(Rf_!F,G)\cong \Hom(F,f^!G)
\]
By applying the fact that left adjoints are cocontinuous one is led to believe, in light of Pitts' theorem, that there should be a cosheaf that realizes the operation of taking derived pushforward with compact supports.
The main result of this section is that we can use the derived equivalence formula~\ref{defn:equiv-fun-projectives} to affirm this belief directly in the cellular world.

In preparation for this result, we describe the coend in the setting of pre(co)sheaves of vector spaces and provide a few examples.

\begin{defn}[Tensoring Sheaves with Cosheaves]
  Let $X$ be a topological space. Given a precosheaf of vector spaces $\cosheaf{G}$ and a presheaf of vector spaces $F$ and a pair of open sets $U\subseteq V$ we can form the following diagram:
  \[
    \xymatrix{ & \hG(V)\otimes F(V) \\ \hG(U)\otimes F(V) \ar[ru]^{r^G_{V,U}\otimes \id} \ar[rd]_{\id\otimes\rho_{U,V}^F} & \\ & \hG(U)\otimes F(U)}
  \]
  By taking the coproduct of such diagrams over all pairs $U\subseteq V$, we define the \define{coend} or \define{tensor product over $X$} of $\cosheaf{G}$ with $F$ to be the following coequalizer:
  \[
  \bigsqcup_{U\subseteq V}\hG(U)\otimes F(V) \rightrightarrows \bigsqcup_W \hG(W)\otimes F(W) \rightarrow \int^{\Open(X)} \hG(W)\otimes F(W)=:\hG\otimes_X F
  \]
\end{defn}

\begin{ex}[Stalks and Skyscraper Cosheaf]
The \define{skyscraper cosheaf} at $x$ is the cosheaf
\[
  \skycshv_x(U)=\left\{ \begin{array}{ll} \Bbbk & \textrm{if $x\in U$}\\
  0 & \text{o.w.}\end{array} \right.
\]
With some thought one can show that the tensor product of any presheaf $F$ with the cosheaf $\skycshv_x$ yields
\[
  \skycshv_x\otimes_X F \cong F_x.
\]
By letting $F$ range over all sheaves one gets a functor
\[
  \skycshv_x\otimes_X - :\Shv(X) \to \Vect \qquad F \rightsquigarrow F_x.
\]
To summarize, we have observed that the operation of taking stalks is equivalent to the process of tensoring with the skyscraper cosheaf.
\end{ex}

\begin{prop}
Let $X$ be the face-relation poset of a cell complex. In the Alexandrov topology on $X$ taking stalks at a cell $\sigma$ is equivalent to tensoring with the elementary projective cosheaf $[\cosheaf{\sigma}]$, i.e.
\[
  [\cosheaf{\sigma}] \otimes_{X} - :\Shv(X) \to \Vect \qquad F \rightsquigarrow F(\sigma).
\]
\end{prop}
\begin{proof}
Observe that the coend restricts the non-zero values of $F$ to the closure of the cell $\sigma$, but this restricted diagram has a terminal object given by $F(\sigma)$ and so the coend returns $F(\sigma)$.
\end{proof}

This allows us to state the main theorem of this section.

\begin{thm}
  Compactly supported cellular sheaf cohomology, as defined in~\ref{subsec:cell-sheaf-cohomology}, is equivalent to tensoring with the image of the constant sheaf through the derived equivalence, i.e.
  \[
    \equivp(k_X)= \bigoplus_{v\in X} [\widehat{v}] \to \bigoplus_{e\in X} [\widehat{e}] \to \bigoplus_{\sigma\in X} [\widehat{\sigma}] \to \cdots.
  \]
\end{thm}
\begin{proof}
  The proof is immediate given the previous description of taking stalks. Observe that
  \[
    \equivp(k_X)\otimes_X F \cong \bigoplus_{v\in X} F(v) \to \bigoplus_{e\in X} F(e) \to \bigoplus_{\sigma\in X} F(\sigma) \to \cdots
  \]
  is by definition quasi-isomorphic to compactly supported cohomology of a cellular sheaf $F$.
\end{proof}

This perspective is especially satisfying for the following reason: it makes transparent how the underlying topology of the space $X$ is coupled with the cohomology of a sheaf $F$. Compactly supported sheaf cohomology arises by tensoring with the complex of cosheaves that computes the Borel-Moore homology of the underlying space.

\newpage

\bibliographystyle{plain}
\bibliography{duality-arxiv}

\end{document}